\documentclass[11pt,a4paper]{elsarticle}
\usepackage[margin=1.5in]{geometry}
\makeatletter
\def\ps@pprintTitle{%
  \let\@oddhead\@empty
  \let\@evenhead\@empty
  \def\@oddfoot{\reset@font\hfil\thepage\hfil}
  \let\@evenfoot\@oddfoot
}
\makeatother
\usepackage{amsmath}
\usepackage{amsthm}
\usepackage{amssymb}
\usepackage{enumerate}
\usepackage{ifthen}
\usepackage{calc}
\usepackage{tikz}
\usetikzlibrary[snakes,matrix,calc]

\usepackage{mathdots}

\newcommand\junk[1]{}
\newcommand\DOI[1]{{\tt DOI:#1}}

\numberwithin{equation}{section}
\newtheorem{theorem}{Theorem}[section]
\newtheorem{lemma}[theorem]{Lemma}
\newtheorem{claim}[theorem]{Claim}

\newtheorem{corollary}[theorem]{Corollary}
\theoremstyle{definition}

\def\qed{{\hfill$\Box$}}
\def\G{\mathbb{G}}

\def\C{\mathcal{C}}
\protected\def\M{\overrightarrow{\mathfrak{M}}}

\newcommand\old[1]{}
\def\B{\mathbb{B}}

\def\bd{\mathbf{d}}
\def\bD{\mathbf{D}}

\def\F{\mathcal{F}}

\begin{document}

\begin{frontmatter}
\title{Efficiently sampling the realizations of irregular, but linearly bounded bipartite and directed degree sequences}

\author{P\'eter L. Erd\H os\fnref{elp}}
\author{Tam\'as R\'obert Mezei\fnref{elp}}
\author{Istv\'an Mikl\'os\fnref{elp}\fnref{miklos}}
\address{Alfr\'ed R{\'e}nyi Institute of Mathematics, Re\'altanoda u
13--15 Budapest, 1053 Hungary\\
    \texttt{email}: $<$erdos.peter,mezei.tamas.robert,miklos.istvan$>$@renyi.mta.hu}
\fntext[elp]{The authors were supported in part by the National Research, Development and Innovation Office --- NKFIH grant K~116769 and KH~126853}
\fntext[miklos]{IM was supported in part by the National Research, Development and Innovation Office --- NKFIH grant SNN~116095.}

\begin{abstract}
Since 1997 a considerable effort has been spent on the study of the \textbf{swap} (switch) Markov chains on graphic degree sequences. Several results were proved on rapidly mixing Markov chains on regular simple, on regular directed, on half-regular directed and on half-regular bipartite degree sequences. In this paper, the main result is the following: Let $U$ and $V$ be disjoint finite sets, and let $0 < c_1 \le c_2 < |U|$ and $0 < d_1 \le d_2 < |V|$ be integers. Furthermore, assume that the bipartite degree sequence on $U \cup V$ satisfies $c_1 \le d(v) \le c_2\ : \ \forall v\in V$  and $d_1 \le d(u) \le d_2 \ : \ \forall u\in U$. Finally assume that $(c_2-c_1 -1)(d_2 -d_1 -1) < 1 + \max \{c_1(|V| -d_2), d_1(|V|- c_2) \}$. Then the swap Markov chain on this bipartite degree sequence is rapidly mixing. The technique applies on directed degree sequences as well, with very similar parameter values.

These results are germane to the recent results of Greenhill and Sfragara about fast mixing MCMC processes on simple and directed degree sequences, where the maximum degrees are  $O(\sqrt{\# \mbox{ of edges}})$. The results are somewhat comparable on directed degree sequences: while the GS results are better applicable for degree sequences developed under some scale-free random process, our new results are better fitted to degree sequences developed under the Erd\H{o}s -- R\'enyi model. For example our results cover all regular degree sequences, the GS model is not applicable when the average degree is $ > n/16.$
\end{abstract}
\end{frontmatter}

\section{Introduction and history}\label{sec:intro}
An important problem in network science is to algorithmically construct typical instances of networks with predefined properties, often expressed as graph measures. In particular, special attention has been devoted to sampling simple graphs with a given degree sequence. In 1997 Kannan, Tetali, and Vempala (\cite{KTV97}) proposed the use of the so-called switch Markov chain approach, which had already been used in statistics. We call this the \textbf{swap Markov chain} approach. We will explain the reason for this discrepancy in a moment.

The \textbf{swap operation} exchanges two disjoint edges $ac$ and $bd$ in the realization $G$ with $ad$ and $bc$ if the resulting configuration $G'$ is again a simple graph (we denote the operation by $ac, bd \Rightarrow ad,bc$). (For details see Section~\ref{sec:def}.) It is a well-known fact that the set of all possible realizations of a \textbf{graphic} degree sequence is connected under this operation. (See, for example, Petersen~\cite{pet} or Havel~\cite{H55} and Hakimi~\cite{H62}.) Similar operations can be considered for bipartite graphs with similar properties. Sometimes not every edge exchange is potentially possible, for example, on a bipartite graph we must ensure that vertices $a$ and $d$ belong to different vertex
classes.

In the literature, the name \textbf{switch operation} is also used, however, in our approach this latter is an operation on integer matrices slightly generalizing the swap operation. (See Section~\ref{sec:anal}.)

The situation is more complicated in case of directed degree sequences. In this case for every vertex the number of incoming edges ({\bf in-degree}) and the number of outgoing edges ({\bf out-degree}) is given in the {\bf degree bi-sequence}. Here the $ac, bd \Rightarrow ad,bc$ type exchange keeps the degree bi-sequence if in both cases $a$ and $b$ are tails of the directed edges. However, imagine that our graph $\vec G$ is a directed triangle $\overrightarrow{C_3}$ while $\vec H$ is the oppositely directed $\overleftarrow{C_3}$. Both graphs have the same degree bi-sequence $\bd=((1,1,1); (1,1,1))$. It is clear that there is only one way to transform the first one into the second one: if we exchange three edges and three non-edges in $\vec G$. We will call this operation a {\bf triple swap} and the ``classical'' one as a {\bf double swap}. (For details see Section~\ref{sec:danal}.) Kleitman and Wang proved in 1973 (\cite{KW73}) that any two realizations of a given graphic degree bi-sequence can be transformed into each other using these two operations. The same fact was re-discovered in 2010 (see~\cite{EMT09}).

The swap Markov chains corresponding to the most common graph models are irreducible, aperiodic, reversible (obey detailed balance), have symmetric transition matrices, and thus have uniform global stationary distributions.

In their paper~\cite{KTV97}, Kannan, Tetali and Vempala conjectured that all these Markov chains are rapidly mixing. The first rigorous proof in this topic is due to Cooper, Dyer and Greenhill about regular simple graphs (\cite{CDG07}). Now, twenty years after the KTV conjecture, we are still far, probably very far, from proving it in its full generality. However, many partial results have been proved; those which play some role in this paper are summarized in the following theorem:

\begin{theorem}\label{th:MCMCk}
The swap Markov chain mixes rapidly for the following degree sequences:
\begin{enumerate}[{\rm (A)}]
\itemsep=0pt
\item $\bd$ is a regular directed degree sequence.
\item $\bd$ is a \textbf{half-regular} bipartite degree sequence.
\item $\bd$ belongs to an \textbf{almost-half-regular} bipartite graph.
\item $\bd $ is an almost-half-regular bipartite degree sequence, where every realization must avoid a fixed (partial) matching.
\item $\bd$ is a directed degree sequence with $2 \le d_{\max} \le \frac{1}{4}\sqrt{M}$, where $M$ is the sum of the degrees, and where the set of all realizations under study is irreducible under the double swap operation.
\end{enumerate}
\end{theorem}
\noindent
The known results on simple degree sequences and bipartite degree sequences use different background machineries. To our knowledge, no successful attempts at adapting any of the machinery to both classes have been made.

The result (A) was proved by Greenhill (\cite{G11}). She used the fact that the set of all realizations of a regular directed degree bi-sequence is irreducible under the double swap operations.
(B) is due to  Mikl\'os, Erd\H{o}s, and Soukup (\cite{MES}). Half-regularity means that in one class the degrees are the same (i.e., regular), while in the other class the only restrictions are those imposed by graphicality.
(C) is due to Erd\H{o}s, Mikl\'os, and Toroczkai (\cite{MCMC-JDM}). Here almost-half-regular means that for any pair of vertices on one side we have $|d(v_1) - d(v_2)| \le 1$.
(D) was proved by Erd\H{o}s, Kiss, Mikl\'os, and Soukup (\cite{EKMS}). This model will be introduced in detail and its intrinsic connection with directed graphs will be fully explained in Section~\ref{sec:danal}. Papers~\cite{G11} and~\cite{EKMS} are using slightly different Markov chains on regular directed degree sequences, therefore (D) does not supersede (A). Finally (E) was proved recently by Greenhill and Sfragara (\cite{G17}). The papers~\cite{BM10} and~\cite{lamar} fully characterize those degree bi-sequences where the set of all realizations is irreducible under the double swap operation.

\medskip\noindent
In this paper, we further extend the set of bipartite degree sequences with rapidly mixing Markov chain processes, using a condition on minimum and maximum degrees. Let $\bd$ be a bipartite degree sequence on the underlying set $U \uplus V$.
\begin{theorem}\label{th:main}
Let $0< c_1\le c_2 < |U|=n$ and $0 < d_1\le d_2 < |V|=m$ be integer parameters and assume that $\bd$ satisfies the following properties:
\begin{align}\label{eq:kicsi}
c_1 \le d(v) \le c_2, & \qquad \forall v \in V \nonumber \\
d_1 \le d(u) \le d_2, & \qquad \forall u\in U.
\end{align}
Furthermore, assume that
\begin{equation}\label{eq:kisebb}
(c_2-c_1-1)\cdot (d_2-d_1-1)\le\max\left\{ c_1(m-d_2), d_1(n-c_2) \right \}
\end{equation}
holds. Then the swap Markov chain on the realizations of this bipartite degree sequence is rapidly mixing.
\end{theorem}
\noindent The proof of this statement is a generalization of the proof of the analogous result on almost-half-regular bipartite degree sequences.

\medskip\noindent Our second main result, similarly to the main new theorem of Greenhill and Sfragara (\cite{G17}), is about directed degree sequences. Let $\vec \bd$ a directed degree sequence on the $n$ element vertex set $V$.
\begin{theorem}\label{th:direct}
Let $0< c_1\le c_2 < n$ and $0 < d_1\le d_2 < n$ be integer parameters and assume that graphic degree be-sequence $\vec\bd$ satisfies the following properties:
\begin{align}\label{eq:dkicsi}
c_1 \le d_\text{out}(v) \le c_2, & \qquad \forall v \in V, \nonumber \\
d_1 \le d_\text{in}(v) \le d_2, & \qquad \forall v\in V.
\end{align}
Furthermore, assume that
\begin{equation}\label{eq:dkisebb}
(c_2-c_1)\cdot (d_2-d_1)\le 2+\max\Big\{ c_1(n-d_2-1)+d_1+c_2,\ \ d_1(n-c_2-1)+c_1+d_2 \Big \}-n
\end{equation}
holds. Then the swap Markov chain, using double and triple swap operations, is rapidly mixing on the realizations of this directed degree sequence.
\end{theorem}
\noindent For details see Section~\ref{sec:danal}.

The proof of our results strongly supports Greenhill's observation about the existing arguments (\cite{G15}): ``In each known case, regularity (or half-regularity) was only required for one lemma, which we will call the critical lemma. This is a counting lemma which is used to bound the maximum load of the flow (see~\cite[Lemma 4]{CDG07},~\cite[Lemma 5.6]{G11},~\cite[Lemma 6.15]{MES})'' --- and some newer examples --- (\cite[Section 3]{G15},~\cite[Lemma 18]{EKMS},~\cite[Lemma 2.5 and Lemma 3.6]{G17}).
The main task is to prove the critical lemmas (Lemma~\ref{th:omega} and~\ref{th:domega}) for our new conditions (\ref{eq:kicsi}-\ref{eq:dkisebb}).
To that end, we first list the fundamental details from~\cite{EKMS}.

\section{Definitions and useful facts}\label{sec:def}
In this section, we recall some well-known definitions and results, furthermore we define our swap Markov chains for the bipartite degree sequences and for the directed degree sequences.

\medskip\noindent
Let $G$ be a simple bipartite graph on $U \uplus V$, where $U=\{ u_1,\ldots, u_n\}$ and $ V=\{v_1,\ldots,v_m\}$, and let its \textbf{bipartite degree sequence} be
\begin{equation}\label{eq:bideg}
\bd(G)=\big (\bd(U), \bd(V)\big )= \Big(\big (d(u_1), \ldots,d(u_n)\bigr), \bigl (d(v_1),\ldots,d(v_m)\bigr)\Big).
\end{equation}
For a valid $ac, bd \Rightarrow ad,bc$ swap operation it is not enough that $ac, bd \in E(G)$ and $ad, bc \not\in E(G)$, we also need that $ad$ can be an edge in some realization. In other words, we need that $a$ and $d$ are in different vertex classes. We will use the name \textbf{chord} for any vertex pair $u,v$ where $uv$ can be an edge in a realization, even if we do not know or do not care whether it is an edge or a non-edge in the current realization. We can reformulate the definition of the swap operation: it can be done if $ac, bd \in E(G)$ and $ad, bc$ are chords.

Now denote by $\G$ the set of all possible realizations of the graphic bipartite degree sequence $\big (\bd(U), \bd(V)\big )$. Consider  two different realizations, $G$ and $H$, of this bipartite degree sequence. As we already mentioned in Section~\ref{sec:intro} it is a well-known fact that the first realization can be transformed into the second one (and vice versa) with a sequence of swap operations. Formally, there exists a sequence of realizations $G=G_0, \dots, G_{i-1},\ G_i=H$, such that for each $j=0,\ldots, i-1$ there exists a swap operation which transforms $G_j$ into $G_{j+1}$. We denote the swap Markov chain as $\mathfrak{M}=(\G,P)$ where the \textbf{transition matrix} $P$ is the following:

In any current realization with probability $\frac{1}{2}$ we stay in the current state (i.e., the chain is lazy) and with probability $\frac{1}{2}$ we choose uniformly two-two vertices $u_1,u_2;v_1,v_2$ from classes $U$ and $V$, respectively. We perform the swap $u_1v_1, u_2v_2\Rightarrow u_1v_2,u_2v_1$ if $u_1v_1,u_2v_2\in E(G)$ and the resulting graph $G'$ is simple. Otherwise we do not perform a move. The swap moving from $G$ to $G'$ is unique, therefore the \textbf{jumping probability} from $G$ to $G'\ne G$ is:
\begin{equation}\label{eq:prob}
\mathrm{Prob}(G\rightarrow G'):= P(G' | G) =
\frac{1}{2\genfrac{(}{)}{0pt}{}{m}{2} \genfrac{(}{)}{0pt}{}{n}{2}}.
\end{equation}
The transition probabilities are time- and edge-independent and are also \textbf{symmetric}. The chain is lazy, therefore aperiodic. It is also reversible, and so its globally stable stationary distribution is uniform.

\medskip\noindent
Now we turn our attention to the notions and notations to deal with Theorem~\ref{th:direct}. Literally this theorem is about directed graphs, however, we will use the machinery developed in the paper~\cite{EKMS}, turning this statement to a theorem about bipartite graphs with some restriction on which edges can be used in the realizations.

Let $\vec G$ be a simple directed graph (parallel edges and loops are forbidden, but oppositely directed edges between two vertices are allowed) with vertex set $X(\vec G) = \{x_1,x_2,  \ldots,x_n \}$ and edge set $E( \vec G)$. For every vertex $x_i\in X$ we associate two numbers: the {\em in-degree\/} and the {\em out-degree\/} of $x_i$. These numbers form the directed degree bi-sequence $\bD$.

We transform the directed graph $\vec G$ into the following {\em bipartite representation\/}: let $B({\vec G})=(U,V;E)$ be a bipartite graph where each class consists of one copy of every vertex from $X(\vec G)$. The edges adjacent to a vertex $u_x$ in class $U$
represent the out-edges from $x$, while the edges adjacent to a vertex $v_x$ in class $V$ represent the in-edges to $x$ (so a directed edge $xy$ corresponds  the edge $u_x v_y$). If a vertex has zero in- (respectively out-) degree in $\vec G$, then we delete the corresponding vertex from $B({\vec G})$. (Actually, this representation is an old trick used by Gale~\cite{G57}, but one can find it already in~\cite{pet}.) The directed degree bi-sequence $\bD$ gives rise to a  bipartite degree sequence.

Here we make good use of the notion of {\em chords\/}: since there are no loops in our directed graph, there cannot be any $(u_x,v_x)$ edge in its bipartite representation --- these vertex pairs are {\bf non-chords}. It is easy to see that these forbidden edges form a forbidden (partial) matching $\F$ in the bipartite graph $B(\vec G)$, or in more general terms, in $B(\bD)$.
To make it easier to remember the nature of restriction, we will denote this \textbf{restricted} bipartite degree sequence with $\vec\bd$.

We consider all realizations $\G(\vec\bd)$ which avoid the non-chords from $\F$. Now it is easy to see that the bipartite graphs in $\G(\vec\bd)$ are in one-to-one correspondence with the possible realizations of the directed degree bi-sequence.

Consider now again our example about two oppositely oriented triangles, $\overrightarrow{C_3}$ and $\overleftarrow{C_3}$. Consider the bipartite representations $B(\overrightarrow{C_3})$ and $B(\overleftarrow{C_3})$, and take their symmetric difference $\nabla$. It contains exactly one alternating cycle (the edges come alternately from $B(\overrightarrow{C_3})$ and $B(\overleftarrow{C_3})$), s.t.\ each vertex pair of distance 3 along the cycle in $\nabla$ forms a non-chord. Therefore, in this alternating cycle no ``classical'' swap can be performed. To address this issue, we introduce a new swap operation: we exchange all edges coming from $B(\overrightarrow{C_3})$ with all edges coming from $B(\overleftarrow{C_3})$ in one operation. The corresponding operation for directed graphs is exactly the triple swap operation.

In general: if the current symmetric difference $\nabla$ contains a length-6 alternating cycle $C_6$ such that all opposite vertex pairs form non-chords, then we allow to perform the corresponding {\bf $C_6$-swap}. In this notation, the original swap should be properly called a {\bf $C_4$-swap} (for obvious reasons), but for the sake of simplicity we only write swap instead of $C_4$-swap. By the constraints posed by the forbidden partial matching, only a subset of all bipartite swaps can be performed. These swaps together with the possible $C_6$-swaps we just defined are called the {\bf $\F$-compatible} swaps or {\bf $\F$-swaps} for short.
\begin{lemma}[\cite{EKM},~\cite{EKMS}]\label{th:F-swaps}
The set $\G(B(\bD))=\G(\vec\bd)$ of all realizations is irreducible under the $F$-swaps.
\end{lemma}
\noindent We are ready to define our swap Markov chain $\M=(\G(\vec\bd), P)$ for the restricted bipartite degree sequence $\vec\bd$.

\medskip\noindent
The transition (probability) matrix  $P$ of the  Markov chain is defined as follows: let the current realization be $G$. Then
\begin{enumerate}[(a)]
\item with probability $1/2$ we stay in the current state, so our Markov chain is lazy;
\item with probability $1/4$ we choose uniformly two-two vertices $u_1,u_2;v_1,v_2$ from classes $U$ and $V$ respectively and perform the swap if it is possible;
\item finally, with probability $1/4$  choose three-three vertices  from $U$ and $V$ and check whether they form three pairs of forbidden chords. If this is the case, then we perform a $C_6$-swap if it is possible.
\end{enumerate}
The swaps moving from $G$ to its image $G'$ is unique, therefore the probability of this transformation (the {\em jumping probability\/} from $G$ to $G'\ne G$) is:
\begin{equation}
\mathrm{Prob}(G \rightarrow_b   G'):= P(G' | G) = \frac{1}{4} \cdot
\frac{1}{\binom{|U|}{2} \binom{|V|}{2}},
\end{equation}
and
\begin{equation}
\mathrm{Prob}(G\rightarrow_c G'):= P(G' | G) = \frac{1}{4} \cdot
\frac{1}{\binom{|U|}{3} \binom{|V|}{3} }.
\end{equation}
(These probabilities reflect the fact, that $G'$ should be derived from $G$ by a regular (that is a $C_4$-)swap or by a $C_6$-swap.) The probability of transforming $G$ to $G'$ (or vice versa) is time-independent and symmetric. Therefore, $P$ is a symmetric matrix, where the entries in the main diagonal are non-zero, but (possibly) distinct values. Our Markov chain is  irreducible (by Lemma~\ref{th:F-swaps}), and it is clearly aperiodic, since it is lazy. Therefore, as it is well-known,  the Markov process $\M$ is reversible with the uniform distribution as the globally stable stationary distribution.

\section{The general properties of the swap Markov chain on bipartite degree sequences}\label{sec:general}
The proofs of our theorems closely follow the proof of Theorem 10 in~\cite{EKMS}, which, in turn, is based on the proof method developed in~\cite{MES}. As we saw earlier the sets of all realizations $\G(\bd)$ and $\G(\vec\bd)$ are slightly different (the latter is somewhat smaller), but the following reasoning from~\cite{EKMS} applies for both. Therefore, the notation $\G$ refers for both realization sets.

Consider two realizations $X, Y\in \G$, and take the symmetric difference $\nabla = E(X) \Delta E(Y)$. Now for each vertex in the bipartite graph $(U,V;\nabla)$ the number of incident $X$-edges ($=E(X)\setminus E(Y)$)  and the number of the incident $Y$-edges are equal. Therefore $\nabla$ can be decomposed into alternating circuits and later into alternating cycles. The way the decomposition is performed is described in detail in Section 5 of the paper~\cite{MES}. Here we just summarize the highlights:

\medskip\noindent
First, we decompose the symmetric difference $\nabla$ into alternating circuits in all possible ways. In each case we get an ordered sequence $W_1,W_2,\dots, W_\kappa$ of circuits. Each circuit is endorsed with a fixed cyclic order.

Now we fix one circuit decomposition. Each circuit $W_i$ in the ordered decomposition has a unique alternating cycle decomposition: $W_i = C^i_1,C^i_2,\dots, C^i_{k_i}$. (This unique decomposition is a quite delicate point and was discussed in detail in Section 5.2 of the  paper~\cite{MES}.)

The ordered circuit decomposition of $\nabla$ together with the ordered cycle decompositions of all circuits provide a well-defined ordered cycle decomposition $C_1,\ldots, C_l$ of $\nabla$. This decomposition does not depend on any swap operations, only on the symmetric difference of realizations $X$ and $Y$.

This ordered cycle decomposition singles out $l-1$ different realizations $H_1,\ldots , H_{l-1}$ from $\G$ with the following property: for each $j=0,\ldots, l-1$ we have $E(H_j) \Delta E(H_{j+1}) = C_{j+1}$ if we apply the notation $H_0=X$ and $H_{l}=Y$. This means that
\[
E(H_i)=E(X)\bigtriangleup\left(\bigcup_{i'\le i}E(C_{i'})\right).
\]

It remains to design a unique canonical path from $X$ to $Y$ determined by the circuit decomposition, which uses the realizations $H_j$ as \textbf{milestones} along the path. In other words, for each pair $H_j, H_{j+1}$ we should design a swap sequence which turns $H_j$ into $H_{j+1}$.

So, the canonical path under construction is a sequence $X=G_0,\dots,G_i,\dots, G_{m}=Y$ of realizations, where each $G_i$ can be derived from $G_{i-1}$ with exactly one swap operation, and there exists an increasing subscript subsequence $0=n_0 <n_1<n_2<\cdots < n_\ell=m$, such that we have $G_{n_k}=H_k$ for every $0\le k\le l$.

\medskip\noindent
We will need the following technical theorem to show that the swap Markov-chain is rapidly mixing.
\begin{theorem}[Section 4 in~\cite{MES}]\label{th:recall}
If the designed canonical path system satisfies the three (rather complicated) conditions below, then the MCMC process is rapidly mixing. The conditions are:
\begin{enumerate}
\itemsep=0pt
\item[$(\Theta)$] For each $i<l$ the constructed path $H_i=G'_0,G'_1, \dots, G'_{m'}=H_{i+1}$ satisfies $m' \le c\cdot |C_{i+1}|$ for a suitable constant $c$.
\item[$(\Omega)$]  $\forall j$ there exists a realization $K_j\in V(\G) $ s.t. $\mathfrak{d} \left (M_X +M_Y -M_{G'_{j}},M_{K_j} \right )$ $\le \Omega_2$, where $M_G$ is the \textbf{bipartite adjacency matrix} of $G$, $\mathfrak{d}$ denotes the Hamming distance, and $\Omega_2$ is a small constant.
\item[$(\Xi)$] For each vertex $G'_{j}$ in the path being traversed the following three objects together uniquely determine the realizations $X, Y$ and the path itself:
    \begin{itemize}
         \itemsep=0pt
         \item The auxiliary matrix $M_X +M_Y -M_{G'_{j}}$,
         \item the symmetric difference $\nabla= E(X) \triangle E(Y)$,
         \item and a polynomial size parameter set $\B$.     \qed
    \end{itemize}
 \end{enumerate}
\end{theorem}
\noindent The meaning of condition $(\Xi)$ is that these structures can be used to control certain features of the canonical path system; namely, their numbers give a bound on the number of canonical paths between any realization pair $X,Y$ which traverse $G'_j$. Condition $(\Omega)$ implies that the space of auxiliary matrices is larger than $V(\G)$ by a multiplicative factor of at most ${(nm)}^{2\Omega_2}$.

In the following sections, we construct the canonical paths among any pair $X,Y$, while ensuring that these three conditions hold.

\section{The construction of swap sequences between consecutive milestones}\label{sec:miles}
\noindent For convenience, we also use the names $G, G' \in \G$ instead of $H_i$ and $H_{i+1}$. These two graphs have almost the same edge set:
\begin{eqnarray*}
\bigl(E(G) \setminus (C_i \cap E(X))\bigr ) \cup (C_i \cap E(Y)) = E(G') \\
\bigl ( E(G') \setminus (C_i \cap E(Y)) \bigr ) \cup (C_i \cap E(X)) = E(G).
\end{eqnarray*}
We  refer to the elements of $C_i \cap E(X)$ as $X$-edges, while the rest of the edges of $C_i$ are $Y$-edges. We denote the cycle $C_i$ by $\C$, which has $2\ell$ edges and its vertices are $u_1,v_1,u_2,v_2,\ldots,u_\ell, v_\ell$. Finally, w.l.o.g.\ we may assume that the chord $u_1v_1$ is a $Y$-edge (and, of course, $v_\ell u_1$  is an $X$-edge).

We will build our canonical path from $G$ towards $G'$. At any particular step, the last constructed realization is denoted by $Z$.  (At the beginning of the process we have $Z=G$.) We are looking for the next realization, denoted by $Z'$.  We will control the canonical path system with an \textbf{auxiliary} structure, originally introduced by Kannan, Tetali and Vempala in~\cite{KTV97}:

The matrix $M_G$ denotes the \textbf{adjacency matrix} of the bipartite realization $G$ where the rows and columns are indexed by the vertices of $U$ and $V$, respectively, with the slight alteration that a position corresponding to a forbidden edge is indicated with a $*$. There is a natural correspondence between the entries of matrices on $U\times V$ and the chords of $G$. Our auxiliary structure is the matrix
\begin{displaymath}
\widehat M(X+Y-Z) = M_X + M_Y - M_Z.
\end{displaymath}
The summation does not change the positions with a $*.$ Since the non-$*$ entries of a bipartite adjacency matrix are $0$  or $1$, the possible entries of $\widehat M$ are $*,-1,0,1,2$. An entry is $*$ if it corresponds to a forbidden edge and it is $-1$ if the edge is missing from both $X$ and $Y$ but it exists in $Z$. It is $2$ if the edge is missing from $Z$ but exists in both $X$ and $Y$. It is $1$ if the edge exists in all three graphs ($X,Y,Z$) or it is there only in one of $X$ and $Y$ but not in $Z$. Finally, it is $0$ if the edge is missing from all three graphs, or the edge exists in exactly one of $X$ and $Y$ and in $Z$.
(Therefore, if an edge exists in exactly one of $X$ and $Y$ then the corresponding chord in $\widehat M$  is always $0$ or $1$.)  It is easy to see that the row and column sums of $\widehat M(X+Y-Z)$ are the same as the row and column sums in $M_X$ (or $M_Y$, or $M_Z$).

\medskip\noindent Now we are ready to determine the $\F$-swap sequence between $G$ and $G'$ and this is the point where realizations from $\G(\bd)$ and $\G(\vec\bd)$ start behave slightly differently. From now on we will work with realizations from $\G(\vec\bd)$ but we will point out those turning points where there are real differences. The first such difference that it can exist a vertex $v_i$ ($1< i < \ell$) such that $u_1v_i$ is a non-chord in case of $G \in \G(\vec\bd)$ while in case of $\in \G(\bd)$ this cannot happen.

\medskip\noindent We determine the $\F$-swap sequence between $G$ and $G'$ from $\G(\vec\bd)$ through an iterative algorithm. At the first iteration we check, step by step, the positions $(u_1, v_2), (u_1, v_3), \ldots, (u_1,v_\ell)$ and take the smallest $j$ for which $(u_1,v_i)$ is an actual edge in $G$. Since $(u_1,v_\ell)$ is an edge in $G$, such an $i$ always exists. A typical configuration is shown in Figure~\ref{f:sweep}.

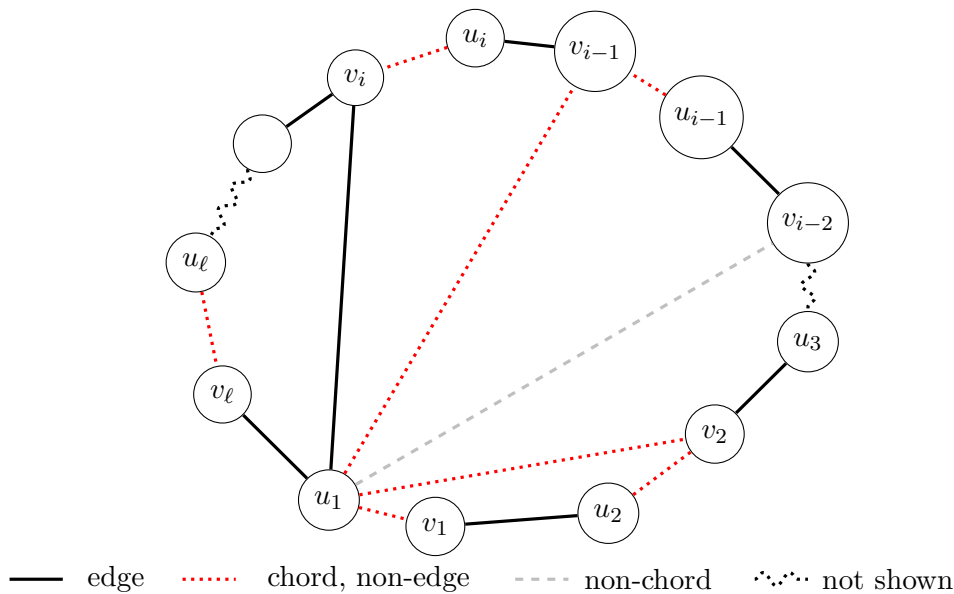
\begin{figure}[h!]
\center{
\begin{tikzpicture}[scale=0.35]
\node at (4,0) [shape=circle,draw] (p1) {$u_1$};
\node at (8,-1) [shape=circle,draw] (p2) {$v_1$};
\node at (14.5,-.5) [shape=circle,draw] (p3) {$u_2$};
\node at (18.5,2.5) [shape=circle,draw] (p4) {$v_2$};
\node at (22,6) [shape=circle,draw] (p17) {$u_3$};
\node at (22,10.5) [shape=circle,draw] (p5) {$v_{i-2}$};
\node at (18,14.5) [shape=circle,draw] (p15) {$u_{i-1}$};
\node at (14,17) [shape=circle,draw] (p16) {$v_{i-1}$};
\node at (9.5,17.5) [shape=circle,draw] (p6) {$u_i$};
\node at (5,16) [shape=circle,draw] (p7) {$v_i$};
\node at (1.5,13.5) [shape=circle,draw] (p8) {$\phantom{u_i} $};
\node at (-1,9) [shape=circle,draw] (p9) {$u_\ell$};
\node at (0,4) [shape=circle,draw] (p10) {$v_\ell$};

\draw [very thick,dotted,red] (p1) -- (p2);
\draw [very thick] (p2) -- (p3);
\draw [very thick, dotted,red] (p3) -- (p4);
\draw [very thick,snake=snake,dotted] (p5) -- (p17);
\draw [very thick] (p4) -- (p17);
\draw [very thick] (p5) -- (p15);
\draw [very thick, dotted,red] (p6) -- (p7);
\draw [very thick, dotted,red] (p15) -- (p16);
\draw [very thick] (p16) -- (p6);
\draw [very thick] (p7) -- (p8);
\draw [very thick,snake=snake,dotted] (p8) -- (p9);
\draw [very thick, dotted,red] (p9) -- (p10);
\draw [very thick]  (p10) -- (p1);
\draw [very thick,dotted,red] (p1) -- (p4);
\draw [very thick,dotted, red] (p1) -- (p16);
\draw [very thick,dashed,gray!50] (p1) -- (p5);
\draw [very thick] (p1) -- (p7);

\draw [very thick] (-8,-3) -- (-6,-3);
\node at (-4,-3) {edge};

\draw [very thick,dotted,red] (-1.5,-3) -- (0.5,-3);
\node at (5.5,-3) {chord, non-edge};

\draw [very thick,dashed,gray!50] (11,-3) -- (13,-3);
\node at (16,-3) {non-chord};

\draw [very thick,snake=snake,dotted] (20,-3) -- (22,-3);
\node at (25,-3) {not shown};

\end{tikzpicture}
}
\caption{Sweeping a cycle}\label{f:sweep}
\end{figure}

We  call the chord $u_1v_i$ the \textbf{start-chord} of the current sub-process and $u_1v_1$ is the \textbf{end-chord}. We will \textbf{sweep} the alternating chords along the cycle. The vertex $u_1$ will be the \textbf{cornerstone} of this operation. This process works from the start-chord $u_1v_i$, $v_i u_i$ (non-edge), $u_i v_{i-1}$ (an edge) toward the end-chord $v_1u_1$ (non-edge) --- switching their status in twos and fours.  We check positions $u_1v_{i-1}, u_1v_{i-2}$ (all are non-edges) and choose the first chord among them, we will call it the \textbf{current-chord}. (Since $u_1$ belongs to at most one non-chord, therefore we never have to check more than two edges to find the first chord, and we need only once to check two.)

\smallskip\noindent {\bf Case 1}: As we just explained the typical situation is that the current-chord is the ``next'' one, so when we start this is typically $u_1w_{i-1}$. Assume that this is a chord. Then we can proceed with the swap operation $v_{i-1}u_i, v_i u_1 \Rightarrow u_1v_{i-1}, u_i v_i$. We just produced the first ``new'' realization in our sequence, this is $G'_1$. For the next swap operation this will be our new current realization. This operation  will be called a \textbf{single-step}.

In a realization $Z$ we call a chord \textbf{bad}, if its state in $Z$ (being edge or non-edge) is different from its state in $G$, or equivalently, different from its state in $G'$, since $G$ and $G'$ differ only on the chords along the cycle $\C$ (recall, that in our nomenclature, a chord is a pair of vertices which may form an edge). After the previous swap, we have two bad chords in $G'_1$, namely $u_1v_{i-1}$ and $v_i u_1$.

Consider now the auxiliary matrix $\widehat M(X+Y-Z)$ (here $Z=G'_1$). As we  saw earlier, any chord not contained in $\C$ has the same state in $X$, $Y$ and $Z$. Accordingly, the corresponding matrix value is $0$ or $1$ in $\widehat M$. We call a position \textbf{bad} in $\widehat M$ if this value is $-1$ or $2$. (A bad position in $\widehat M$ always corresponds to a bad chord.) Since we switch the start-chord into a non-edge, it may become $2$ in $\widehat M$ (in case the start-chord is an edge in both $X$ and $Y$). On the other hand, the current-chord turned into an edge. If it is a non-edge in both $X$ and $Y$ then its corresponding value in $\widehat M$ becomes $-1$. After this step, we have at most two bad positions in the matrix, at most one with $2$-value and at most one with $-1$-value. Finishing our swap operation, the previous current-chord becomes the new start-chord, so it is the edge $u_1,v_{i-1}$.

\smallskip\noindent {\bf Case 2}: If the position below  the start-chord (this is now $u_1v_{i-2}$) is a non-chord, then we cannot produce the previous swap. Then the non-edge $u_1v_{i-3}$ is the current-chord. For sake of simplicity we assume that $i-3=2$ so we are in Figure~\ref{f:sweep}. (That is, $i-1=4$.) Consider now the alternating  $C_6$ cycle:  $u_1,v_2, u_3,v_3, u_4,v_4$. It has altogether three vertex pairs which may be used to perform an $\F$-swap operation. We know already that $u_1v_3$ is a non-chord.
If neither $v_2u_4$ nor $u_3v_4$ are chords, then this alternating cycle provides an $\F$-compatible circular $C_6$-swap. Again, we found the valid swap $v_2u_3, v_3u_4, v_4u_1 \Rightarrow u_1v_2, u_3v_3, u_4v_4$. After that we again have 2 bad chords, namely $u_1v_2$ and $v_4 u_1$, and together we have at most two bad positions in the new $\widehat M(X+Y-Z)$ with at most one $2$-value and at most one $-1$-value.

Finally, if one position, say $v_2u_4$, is a chord then we can process this $C_6$ with two swap operations. If this chord is, say, an actual edge, then we swap $v_2u_4, v_4u_1 \Rightarrow u_1v_2, u_4v_4$. After this we can take care for the $v_2,u_3,v_3,u_4$ cycle. Along this sequence we never create more than 3 bad chords: the first swap makes chords $v_2u_4$, $v_4u_1$, and $u_1v_2$ bad ones, and the second cures $v_2u_4$ but does not touch $u_1v_2$ and $v_4u_1$. So, along this swap sequence we have 3 bad chords, and in the end we have only 2.
On the other hand, if the chord $v_2u_4$ is not an edge, then we can swap $v_2u_3, v_3u_4 \Rightarrow u_3v_3, u_4v_2$, creating one bad edge, then taking care the four cycle $u_1,v_2,u_4,v_4$ we cure $v_2u_4$ but we switch $u_1v_2$ and $v_4u_1$ into bad chords. We finished our \textbf{double-step} along the cycle.

In a double-step we create at most three bad chords. When the first swap uses three chords along the cycle then we may have at most one bad chord (with $\widehat M$-value $0$ or $-1$) and then the next swap switches back the chord into its original status, and makes two new bad chords (with at most one $2$-value and one $-1$-value). When the first swap uses only one chord from the cycle, then it creates three bad chords (changing two chords into non-edge and one into edge), therefore it may create at most two $2$-values and one $-1$-value. After the second swap, there will be only two bad chords, with at most one $2$-value, and at most one $-1$-value.

When only the third position corresponds to a chord in our $C_6$ then after the first swap we may have two $-1$-values and one $2$-value. However, again after the next swap we will have at most one of both types.

\medskip\noindent
Finishing our single- or double-step the previous current-chord becomes the new start-chord and we look for the new current-chord. Then we repeat our procedure. There is only one important point to be mentioned: along the step, the start-chord switches back into its original status, so it will not be a bad chord anymore. So even if we face a double-step the number of bad chords never will be bigger than three (together with the chord $v_i u_1$ which is still in the wrong status, so it is bad), and we have always at most two $2$-values and at most one $-1$-value in $\widehat M(X+Y-Z)$.

When our current-chord becomes to $v_1u_2$ then the last step will switch back the last start-chord into its correct status, and the last current-chord cannot be in bad status. So, when we finish our sweep from $u_1v_i$ to $v_1u_1$ at the end we will have only one bad chord (with a possible $2$-value in $\widehat M$). This concludes the first iteration of our algorithm.

\bigskip\noindent
For the next iteration, we seek a new start-chord between $v_i u_1$ and $v_\ell u_1$  and chord $v_i u_1$ becomes the new end-chord. We repeat our sweeping process until there are no more unprocessed chords. Upon completion, we find a realization sequence from $G$ to $G'$. If in the first sweep we had a double-step, then such a step will never occur later, so altogether with the (new) bad end-chord we never have more than three bad chords (corresponding to at most two $2$-values and at most one $-1$-value).

However, if the double-step occurs sometimes later, for example in the second sweep, then we face to the following situation: if we perform a circular $C_6$-swap, then there cannot be any problem. Thus, we may assume that there is a chord in our $C_6$, suitable for a swap. If this chord is a non-edge, then the swap around it produces one bad chord, and at most one bad position in $\widehat M$. The only remaining case when that chord is an edge. After the first swap there will be four bad chords, and there may be at most three $2$-values and at most one $-1$ value. However, after the next swap (finishing the double step) we annihilate one of the $2$-values, and after that swap there are at most two $2$-values and at most $-1$-value along the entire swap sequence. When we finish our second sweep, then chord $v_i u_1$ will be switched back into its original status, it will not be bad anymore.

We apply iteratively the same algorithm, and after at most $\ell$ sweep sequence, we will process the entire cycle $\C$. This finishes the construction of the required $\F$-swap sequence (and the required realization sequence). \qed

\medskip\noindent Meanwhile we also proved the following important observations:
\begin{lemma}\label{th:bad}
For the Markov chain $\mathfrak{M}$, along our procedure we always have at most two $2$-values and at most one $-1$-value in our auxiliary matrix $\widehat M(X+Y-Z)$.
\end{lemma}
\begin{lemma}\label{th:dbad}
For the Markov chain $\M$, along our procedure each occurring auxiliary matrix $\widehat M(X+Y-Z)$ is at most swap-distance one from a matrix with at most three bad positions: with at most two $2$-values and with at most one $-1$-value in the same column.
\end{lemma}
It remains to show that the defined swap sequences between $H_i$ and $H_{i+1}$, using the cornerstone $u_1$ chosen in ($\Phi$), satisfy conditions $(\Theta)$, $(\Omega)$, and $(\Xi)$ of Theorem~\ref{th:recall}. The first one is easy to see, since we can process any cycle of length $2\ell$ in $\ell-1$ swaps. Therefore, we may choose $c=1$ in $(\Theta)$.

In Theorem~\ref{th:recall} condition ($\Xi$) heavily depends on parameter $\B$. However, here the sweeping process is very similar to the one used in paper~\cite{EKMS} and the same upper bound applies to it. Therefore, we also have that $\B$ is bounded by a polynomial of small exponent. So only condition ($\Omega$) remains to be checked.

\medskip\noindent
Until this very moment the choice of the cornerstone vertex $u_1$ was arbitrary. Before we turn to the analysis of the swap sequences, we choose which particular vertex of the cycle $\C$ will serve as its cornerstone.

Let the submatrix $A$ contain those positions from any adjacency or any auxiliary matrix which corresponds the positions $u_i v_j$ defined by the vertices from $\C$. Furthermore, denote by $A[Z]$ the submatrix of $\widehat M(X+Y-Z)$ spanned by the vertices of $\C$. Then:
\begin{enumerate}
\itemsep=0pt
\item[$(\Phi)$] Let $u_1$ be a vertex which has the lowest row sum in the submatrix $A[H_i]=A[G]$.
\end{enumerate}

\section{The analysis of the swap sequences between milestones in $\mathfrak{M}$}\label{sec:anal}
\noindent
In this section, we will analyze the undirected case. Before we continue, let us recall the main characteristic of the bipartite degree sequence $\bd$: we have integers $0 < c_1 \le c_2 < n$ and $0 < d_1 \le d_2 < m$ with the properties $c_1 \le d(v) \le c_2, \quad \forall v \in V$ and $d_1 \le d(u) \le d_2, \quad \forall u\in U$. Furthermore, we know that the parameters satisfy the inequality (\ref{eq:kisebb}).

Now we introduce the new \textbf{switch} operation on $0/1$ matrices: we fix the four corners of a submatrix, and we add $1$ to two corners in a diagonal, and add $-1$ to the corners on the other diagonal. This operation clearly does not change the column and row sums of the matrix. For example, if we consider the matrix $M_G$ of a realization of $\bd$ and make a valid swap operation, then this is equivalent to a switch in this matrix. The next statement is trivial but very useful:
\begin{claim}\label{th:Hamming}
If two  matrices have \textbf{switch-distance} 1, then their Hamming distance is $4$. Consequently, if the switch-distance is $c$ then the Hamming distance is bounded by $4c$.
\end{claim}
\noindent The next lemma shows that property $(\Omega)$ holds for the auxiliary matrices along the swap sequence from $G$ toward $G'$:
\begin{lemma}\label{th:omega}
For any realization $Z$ along the constructed swap sequence from $G$ to $G'$ in $\G(\bd)$ there exists a realization $K=K(Z)$ such that
\begin{displaymath}
\mathfrak{d}\left (\widehat M(X+Y-Z), M_K \right ) \le 16.
\end{displaymath}
\end{lemma}
\begin{proof}
The swap sequence transforming $G$ to $G'$ only touches chords induced by $V(\C)$. Therefore, the row and column sums in $A[Z]$ are the same as that of $A[G]$, so the cornerstone has the minimum row sum in $A[Z]$ as well.

\medskip

Any entries of $2$'s and $-1$'s in $\widehat M$ are in the row of $u_1$, moreover, they are contained in $A[Z]$. Suppose $\widehat M_{u_1,v_j}=2$. The sum of entries of $A[Z]$ in the column $v_j$ is $< |U\cap V(\C)|= |V \cap V(\C)|$, therefore $\exists \ u_k\in U\cap V(\C)$ such that $\widehat M_{u_k,v_j}=0$.
Since the sum of the entries in row $u_1$ is minimum among the rows of $A[Z]$, there must $\exists v_l\in V\cap V(\C)$ such that $\widehat M_{u_k,v_l}>\widehat M_{u_1,v_l}$. Obviously, $\widehat M_{u_k,v_l}<2$, so $\widehat M_{u_1,v_l}\in \{0,-1\}$. The switch operation  $u_1v_j, u_k v_l\Rightarrow u_1v_l, u_k v_j$ (decrease the entries of the matrix by one at positions $u_1 v_j$ and $u_k v_l$, and increase the entries at positions $u_1 v_l$ and $u_k v_j$ by one) in $\widehat M$ (and in $A[Z]$) eliminates the entry of 2 at $u_1v_j$, and creates an entry of 1 at both  $u_1v_j$ and $u_k v_j$. In the column $v_l$ three scenarios are possible:  either the entry $-1$ and a $0$ exchange their positions, or a $0$ and a $1$ exchange their positions; finally, it is also possible that the $-1$ and a $1$ both become $0$.

By repeating the previous argument, we may eliminate one more entry 2, if necessary, from $A[Z]$. (Recall that at the beginning we had at most two 2s in $\widehat M$.) Although it is possible that the entry $-1$ is not in the $u_1$-row anymore, it does not cause any hardship. Let $\widehat M'$ be the matrix we get after performing these at most two switches that eliminate the 2's. Each entry of $\widehat M'$ is a 0 or a 1, except at most one $-1$ entry.

\bigskip\noindent
From now on we will consider the entire matrix $\widehat M'$ and not only $A$. Suppose that $\widehat M'_{u_0,v_0}=-1$. Let $U'=\{u\in U\ |\ \widehat M'_{u,v_0}=1\}$ and $V'=\{v\in V\ |\ \widehat M'_{u_0,v}=1\}$.
If $\exists (u,v)\in U'\times V'$ such that $\widehat M'_{u,v}=0$, then switch operation $u_0 v,u v_0\Rightarrow u_0v_0,uv$ transforms $\widehat M'$ into an adjacency matrix.

\medskip

Suppose from now on, that $\forall (u,v)\in U'\times V'$ we have $\widehat M'_{u,v}=1$. Let
\begin{align*}
  U''&=\{ u\in U\ |\ \exists v\in V'\ : \ \widehat M'_{u,v}=0\}, \\
  V''&=\{ v\in V\ |\ \exists u\in U'\ : \ \widehat M'_{u,v}=0\}.
\end{align*}
Clearly, $U''\cap U'=V''\cap V'=\emptyset$. Suppose there $\exists (u_2,v_2)\in U''\times V''$ such that $\widehat M'_{u_2,v_2}=1$. By definition, there $\exists (u_1,v_1)\in U'\times V'$ such that $\widehat M'_{u_2,v_1}=0$ and $\widehat M'_{u_1,v_2}=0$. Clearly, applying first the switch operation $u_1,u_2$ and $v_1,v_2$, and then the operation $u_0,u_1$ and $v_0,v_1$ transforms $\widehat M'$ into an adjacency matrix.

\medskip

Lastly, suppose that $\forall (u,v)\in U''\times V''$ we have $\widehat M'_{u,v}=0$. This case is shown in Figure~\ref{fig:minusone}.

\begin{figure}[h!]
\begin{center}
  \begin{tikzpicture}

  \tikzstyle{row 1}=[font=\bfseries]
  \tikzstyle{column 1}=[font=\bfseries]
  \matrix [matrix of nodes,left delimiter=(,right delimiter=), matrix anchor=north west] at (0,0)
  {
  |(corner)| -1 & |(v1)| 1 & $\cdots$ & |(v2)| 1 & 0 & $\cdots$ & 0 & |(v3)| 0 & $\cdots$ & |(v4)| 0 \\
  |(u1)| 1 & |(nw1)| & & & & & & |(nw2)| & & \\
  $\vdots$ \\
  |(u2)| 1 & & & |(se1)| & & & |(se6)| \\
  0 & & & & |(nw5)|\\
  $\vdots$ \\
  0 & & & |(se7)| & & & |(se5)| & & & |(se2)| \\
  |(u3)| 0 & |(nw3)| & & & & & & |(nw4)| & & \\
  $\vdots$  \\
  |(u4)| 0 & & $\cdots$ & & & & |(se3)| & & & |(se4)| \\
  };

  \node (nw6) at (nw1) {};
  \node (nw7) at (nw1) {};

  \foreach \x/\y in {6/{},7/{},1/{1},2/{0/1},3/{0/1},4/{0},5/{0/1}}
  {
    \filldraw[black!20] ($(nw\x)-(0.15,-0.3)$) rectangle ($(se\x)+(0.15,0)$);
    \node[at=($0.5*(nw\x)+0.5*(se\x)$), anchor=center] {\huge\bfseries \y};
  }

  \draw [decorate,thick,decoration={brace,amplitude=5pt}]
    ($(u2)-(1,0)$) -- ($(u1)-(1,0)$);
  \draw [decorate,thick,decoration={brace,amplitude=5pt}]
    ($(u4)-(1,0)$) -- ($(u3)-(1,0)$);

  \draw[anchor=west] ($(corner)-(1,0)$) node[anchor=east] {$u_1$};
  \draw[anchor=east] ($0.5*(u1)+0.5*(u2)-(2,0)$) node[anchor=west] {$U'$};
  \draw[anchor=east] ($0.5*(u3)+0.5*(u4)-(2,0)$) node[anchor=west] {$U''$};

  \draw [decorate,thick,decoration={brace,amplitude=5pt}]
    ($(v1)+(0,0.5)$) -- ($(v2)+(0,0.5)$);
  \draw [decorate,thick,decoration={brace,amplitude=5pt}]
    ($(v3)+(0,0.5)$) -- ($(v4)+(0,0.5)$);

  \draw[anchor=north] ($(corner)+(0,1)$) node {$v_1$};
  \draw[anchor=center] ($0.5*(v1)+0.5*(v2)+(0,1)$) node {$V'$};
  \draw[anchor=center] ($0.5*(v3)+0.5*(v4)+(0,1)$) node {$V''$};

  \end{tikzpicture}
\end{center}
\caption{$\widehat M'$ is shown; each of the entries in the regions marked with 0/1 may be 0 or 1.}\label{fig:minusone}
\end{figure}

In addition to the zeroes in $U''\times V''$, $\widehat M'_{u,v_0}=0$ for any $u\in U''$. We have
\begin{align}
\label{eq:m1}
\begin{split}
  |U''|&\cdot (m-d_1) \ge\left|\left\{ (u,v)\in U''\times V\ \big|\ \widehat M'_{u,v}=0\right\}\right|=\\
  &=|U''\times V''|+|U''|+\left|\left\{ (u,v)\in U''\times (V\setminus V''\setminus\{v_0\}) \big|\ \widehat M'_{u,v}=0\right\}\right|.
\end{split}
\end{align}
The right-hand side can be estimated from below as follows. Since the row and column sums of $\widehat M'$ are the same as that of $M_X$, we have
\begin{displaymath}
  |U'|\ge c_1-\widehat M'_{u_0,v_0}=c_1+1, \quad \text{and} \quad
  |V'|\ge d_1-\widehat M'_{u_0,v_0}=d_1+1. \\
\end{displaymath}
% Consequently,
% \begin{align*}
%   |U\setminus U'\setminus U''|&\le n-c_1-1-|U''|, \\
%   |V\setminus V'\setminus V''|&\le n-d_1-1-|V''|.
% \end{align*}
For any $v\in V'$ and $u\in U\setminus U''$, we have $\widehat M'_{u,v}=1$. Also, for any $u\in U'$ and $v\in V\setminus V''$, we have $\widehat M'_{u,v}=1$. Therefore
\begin{align*}
  n-c_1-2\ge |U\setminus U'\setminus \{u_0\}|\ge &|U''|\ge n-c_2, \\
  m-d_1-2\ge |V\setminus V'\setminus \{u_0\}|\ge &|V''|\ge m-d_2.
\end{align*}
Clearly, if $c_2\le c_1+1$ or $d_2\le d_1+1$ (i.e., $G$ is almost half-regular), we already have a contradiction. We also have
\begin{align}
\label{eq:m3}
\begin{split}
  &\left|\left\{ (u,v)\in U''\times (V\setminus V''\setminus \{v_0\}) \big|\ \widehat M'_{u,v}=0\right\}\right|\ge \\
  &\ge (n-c_2)(m-|V''|-1)-\left|U\setminus U'\setminus U''\right|\cdot \left|V\setminus V'\setminus V''\setminus \{v_0\}\right|\ge \\
  &\ge (n-c_2)(m-|V''|-1)-\left(n-c_1-1-|U''|\right)\cdot \left(m-d_1-2-|V''|\right). \\
\end{split}
  % \ge \\
  %&\ge (n-c_2)(m-|V''|-1)-\left(n-c_1-1-|U''|\right)\cdot\left( m-d_1-1-|V''|-1\right).
\end{align}
 Combining Equation~\ref{eq:m1} and~\ref{eq:m3},
\begin{align*}
  |U''|\cdot (m-d_1) &\ge
  %&\ge |U''\times V''|+|U''|+(n-c_2)(m-|V''|-1)-\left(n-c_1-1-|U''|\right)\cdot\left( m-d_1-1-|V''|-1\right)= \\
  %&=(n-c_2)\cdot (m-|V''|-1)+(n-c_1-1)+|U''|(m-d_1-1)+|V''|(n-c_1-1)-\\
  %&\quad -(n-c_1-1)\cdot(m-d_1-1)= \\
  (n-c_2)\cdot (m-|V''|-1)+|U''|(m-d_1-1)+\\ & \ \quad +(|V''|+1)\cdot(n-c_1-1)- (n-c_1-1)\cdot(m-d_1-1).
\end{align*}
Further simplifying:
\begin{align*}
  %a(m-d_1')+a&\ge (n-c_2)\cdot (m-b)+(c_2-c_1')+a(m-d_1')+b(n-c_1')-(n-c_1')\cdot(m-d_1') \\
  %c_1'-c_2+a&\ge -c_2m-bn+bc_2+bn-bc_1'+c_1'm+d_1'n-c_1'd_1' \\
  |U''|&\ge (|V''|+1)(c_2-c_1-1)+(c_1+1-c_2)m+(d_1+1)n-(c_1+1)(d_1+1).
\end{align*}
Since we may suppose that $c_2\ge c_1+2$, we can substitute $|V''|\ge m-d_2$ and $|U''|\le n-c_1-2$ into the inequality, yielding
\begin{align*}
  %n-c_2-1&\ge c_2m-c_1'm-c_2d_2+c_1'd_2-c_2m+c_1'm+d_1'n-c_1'd_1' \\
  %n-c_2-1&\ge -c_2d_2+c_1'd_2+d_1'n-c_1'd_1' \\
  %n-c_2-1&\ge -(c_2-c_1')(d_2-d_1')+d_1'(n-c_2) \\
  (c_2-c_1-1)(d_2-d_1-1) &\ge d_1(n-c_2)+1.
\end{align*}
  Symmetrically, a similar derivation gives
\[ (c_2-c_1-1)(d_2-d_1-1) \ge c_1(n-d_2)+1. \]
The last two inequalities clearly contradict the assumptions of this claim.

\medskip

In summary, in every case there exist at most 4 switches which transform $\widehat M$ into a $0-1$ matrix, which is a matrix with suitable row- and column sums, therefore it is the adjacency matrix of a realization $K$ of the degree sequence $\bd$.
\end{proof}

\section{The analysis of the swap sequences between milestones in $\M$}\label{sec:danal}
\noindent Now we turn to the directed case. As in the previous section, condition $(\Omega)$ is the only remaining assumption of Theorem~\ref{th:recall} which does not immediately follow from Section~\ref{sec:miles}.

\begin{lemma}\label{th:domega}
For any realization $Z$ along the constructed swap sequence from $G$ to $G'$ there exists a realization $K=K(Z)$ such that
	\begin{displaymath}
	\mathfrak{d}\left(\widehat M(X+Y-Z), M_K \right) \le 20.
	\end{displaymath}
\end{lemma}
\begin{proof}
As described by Lemma~\ref{th:dbad}, it is possible that realization $Z$ is derived by an $\F$-swap which is a first $C_4$-swap to resolve an alternating $C_6$ cycle along the sweep. It may introduce an extra 2-value and/or a $-1$-value into the auxiliary structure. But this Lemma also shows that the next $C_4$ swap will revert these extra bad positions. Therefore let $Z^S$ denote the realization $Z$ itself if this extra swap is not needed, or the new realization if it is needed. Then $\widehat M(X+Y-Z^S)$ has at most two entries of 2 and at most one entry of $-1$. Now we have to show that there is a realization $K$ such that
\begin{displaymath}
\mathfrak{d}\left( \widehat M(X+Y-Z^S), M_K \right) \le 16.
\end{displaymath}
As before we will use the shorthand $\widehat M(X+Y-Z^S)=\widehat M.$

The swap sequence transforming the bipartite representation $G$ to $G'$ (also, the previous extra swap) only touches chords induced by $V(\C)$. Therefore, the row and column sums in $A[Z]$ are the same as that of $A[G]$, so the cornerstone has the minimum row sum in $A[Z]$ as well.

\medskip

Any entries of $2$'s and $-1$'s in $\widehat M$ are in the row of $u_1$, moreover, they are contained in $A[Z]$. Suppose $\widehat M_{u_1,v_j}=2$. The column of $v_j$ in $A[Z]$ contains at least one zero, therefore there exist two vertices $u_k,u_{k'}\in U\cap V(\C)$ such that $\widehat M_{u_k,v_j}=0$ and $\widehat M_{u_{k'},v_j}=0$, even if there is a $*$ in the column of $v_j$.
	We have two cases.
	\begin{enumerate}
		\item There $\exists v_l\in V\cap V(\C)$ such that $\widehat M_{u_k,v_l}>\widehat M_{u_1,v_l}$: obviously, $\widehat M_{u_k,v_l}<2$, so $\widehat M_{u_1,v_l}\in \{0,-1\}$.
    The switch operation  $u_1v_j, u_k v_l\Rightarrow u_1v_l, u_k v_j$ (decrease the entries of the matrix by one at positions $u_1 v_j$ and $u_k v_l$, and increase the entries at positions $u_1 v_l$ and $u_k v_j$ by one) in $\widehat M$ (and in $A[Z]$) eliminates the entry of 2 at $u_1v_j$, and creates an entry of 1 at both $u_1v_j$ and $u_k v_j$.
    In the column $v_l$ three scenarios are possible:  either the entry $-1$ and a $0$ exchange their positions, or a $0$ and a $1$ exchange their positions; finally, it is also possible that the $-1$ and a $1$ both become $0$.
		\item If for all $v_l\in V\cap V(\C)$ either $\widehat M_{u_k,v_l}\le \widehat M_{u_1,v_l}$ or $\widehat M_{u_k,v_l}=*$ or $\widehat M_{u_1,v_l}=*$ holds: since the sum of the entries in row $u_1$ is minimum among the rows of $A[Z]$, this is only possible if there exist $v_{l'},v_{l''}\in V\cap V(C)$ such that $\widehat M_{u_1,v_{l'}}=*$, $\widehat M_{u_k,v_{l'}}=1$, $\widehat M_{u_1,v_{l''}}=-1$, $\widehat M_{u_k,v_{l''}}=*$, and for $v_l\in V\cap V(\C)\setminus \{v_j,v_{l'},v_{l''}\}$ we have $\widehat M_{u_k,v_l}= \widehat M_{u_1,v_l}$. This is shown on Figure~\ref{fig:twostars}.
		\begin{figure}[h!]
			\begin{center}
				\begin{tikzpicture}

				\tikzstyle{row 1}=[font=\bfseries]
				\tikzstyle{row 3}=[font=\bfseries]
				\matrix [matrix of nodes,left delimiter=(,right delimiter=), matrix anchor=north west] at (0,0)
				{
					|(u1)| $\cdots$ & |(vj)| 2 & $\cdots$ & |(vlp)| $*$ & $\cdots$ & |(vlpp)| -1 & $\cdots$ \\
					& & & $\vdots$ & & & \\
					|(uk)| $\cdots$ & 0 & $\cdots$ & 1 & $\cdots$ & $*$ & $\cdots$ \\
					& & & $\vdots$ & & & \\
				};

				\draw[anchor=west] ($(u1)-(1,0)$) node[anchor=east] {$u_1$};
				\draw[anchor=west] ($(uk)-(1,0)$) node[anchor=east] {$u_k$};

				\draw[anchor=north] ($(vj)+(0,1)$) node {$v_j$};
				\draw[anchor=north] ($(vlp)+(0,1)$) node {$v_{l'}$};
				\draw[anchor=north] ($(vlpp)+(0,1)$) node {$v_{l''}$};

				\end{tikzpicture}
			\end{center}
			\caption{$\widehat M$ is shown for Case 2 of the proof of Lemma~\ref{th:domega}}\label{fig:twostars}
		\end{figure}
	\end{enumerate}
	If the second case applies, the first case must hold if we replace $k$ by $k'$; if not, column $v_{l''}$ would contain two $*$, a contradiction.

	\bigskip\noindent
	By repeating the previous argument, we may eliminate one more entry 2, if necessary, from $A[Z]$ (and $\widehat M$). (Recall that at the beginning we had at most two 2's in $\widehat M$.) Although it is possible that the entry $-1$ is not in the $u_1$-row anymore, it does not cause any hardship. Let $\widehat M'$ be the matrix we get after performing these at most two switches that eliminate the 2's. Each entry of $\widehat M'$ is a 0 or a 1, except at most one $-1$ entry.

	\bigskip\noindent
	From now on we will consider the entire matrix $\widehat M'$ and not only $A$. Suppose that $\widehat M'_{u_0,v_0}=-1$. Let $U'=\{u\in U\ |\ \widehat M'_{u,v_0}=1\}$ and $V'=\{v\in V\ |\ \widehat M'_{u_0,v}=1\}$.
	If $\exists (u,v)\in U'\times V'$ such that $\widehat M'_{u,v}=0$, then switch operation $u_0 v,u v_0\Rightarrow u_0v_0,uv$ transforms $\widehat M'$ into an adjacency matrix.

	\medskip

	Suppose from now on, that $\forall (u,v)\in U'\times V'$ we have $\widehat M'_{u,v}\in \{1,*\}$. Let
	\begin{align*}
	U''&=\{ u\in U\ |\ \exists v\in V'\ : \ \widehat M'_{u,v}=0\}, \\
	V''&=\{ v\in V\ |\ \exists u\in U'\ : \ \widehat M'_{u,v}=0\}.
	\end{align*}
	Clearly, $U''\cap U'=V''\cap V'=\emptyset$. Suppose $\widehat M'[U'', V'']$ contains more than $|U''|$ entries of 1's. A simple pigeon-hole principle argument implies that there $\exists (u_2,v_2)\in U''\times V''$ and $\exists (u_1,v_1)\in U'\times V'$ such that $\widehat M'_{u_2,v_2}=1$, $\widehat M'_{u_2,v_1}=0$, $\widehat M'_{u_1,v_2}=0$, and $M'_{u_1,v_1}=1$. Clearly, applying first the switch operation $u_1,u_2$ and $v_1,v_2$, and then the operation $u_0,u_1$ and $v_0,v_1$ transforms $\widehat M'$ into an adjacency matrix.

	\medskip

	Lastly, suppose that $\widehat M'[U'', V'']$ contains at most $|U''|$ entries of 1's. We have
	\begin{align*}
	&|U''|\cdot (n-d_1)-|U''| \ge \\
	&\ge\left|\left\{ (u,v)\in U''\times V\ \big|\ \widehat M'_{u,v}=0\right\}\right|\ge\\
	&\ge \left(|U''\times V''|-|U''|\right)+\left|\left\{ (u,v)\in U''\times (V\setminus V'') \big|\ \widehat M'_{u,v}\in \{0,*\}\right\}\right|-|U''|,
	\end{align*}
	since $\widehat M'[U'', V]$ contains exactly $|U''|$ of $*$.
	The right-hand side can be estimated from below as follows. Since the row and column sums of $\widehat M'$ are the same as that of $M_X$, we have
	\begin{displaymath}
	|U'|\ge c_1-\widehat M'_{u_0,v_0}=c_1+1, \quad \text{and} \quad
	|V'|\ge d_1-\widehat M'_{u_0,v_0}=d_1+1. \\
	\end{displaymath}
	% Consequently,
	% \begin{align*}
	%   |U\setminus U'\setminus U''|&\le n-c_1-1-|U''|, \\
	%   |V\setminus V'\setminus V''|&\le n-d_1-1-|V''|.
	% \end{align*}
	%For any $v\in V'$ and $u\in U\setminus U''$, we have $\widehat M'_{u,v}=1$. Also, for any $u\in U'$ and $v\in V\setminus V''$, we have $\widehat M'_{u,v}=1$.
	Also,
	\begin{align*}
	n-c_1-2\ge |U\setminus U'\setminus \{u_0\}|\ge &|U''|\ge n-c_2-1, \\
	n-d_1-2\ge |V\setminus V'\setminus \{u_0\}|\ge &|V''|\ge n-d_2-1.
	\end{align*}
	Clearly, if $c_2=c_1$ or $d_2=d_1$ (i.e., $G$ is half-regular), we already have a contradiction. Each column in $V\setminus V''$ may contain at most one $*$, therefore
	\begin{align*}
	&\left|\left\{ (u,v)\in U''\times (V\setminus V''\setminus \{v_0\}) \big|\ \widehat M'_{u,v}\in \{0,*\}\right\}\right|\ge \\
	&\qquad\qquad\ge (n-c_2)(n-|V''|-1)-\left|U\setminus U'\setminus U''\right|\cdot \left|V\setminus V'\setminus V''\setminus \{v_0\}\right|.
	% \ge \\
	%&\ge (n-c_2)(m-|V''|-1)-\left(n-c_1-1-|U''|\right)\cdot\left( m-d_1-1-|V''|-1\right).
	\end{align*}
	Moreover, $\widehat M'_{u,v_0}\in \{0,*\}$ for any $u\in U''$. Combining these inequalities, we get
	\begin{align*}
		|U''|\cdot (n-d_1)\ge  &\left(|U''\times V''|-|V''|\right)+|U''|+\\ &+(n-c_2)(n-|V''|-1)-\left|U\setminus U'\setminus U''\right|\cdot \left|V\setminus V'\setminus V''\setminus \{v_0\}\right|.
	\end{align*}
	A few lines of computation similar to those in Section~\ref{sec:anal} give
	\begin{align*}
		(c_2-c_1)(d_2-d_1)\ge d_1(n-c_2-1)+3+c_1+d_2-n.
	\end{align*}
	Symmetrically, we also have
	\[ 	(c_2-c_1)(d_2-d_1)\ge c_1(n-d_2-1)+3+d_1+c_2-n. \]
	The last two inequalities clearly contradict the assumptions of this claim.

	\medskip

	In summary, in every case there exist at most 4 switches which transform $\widehat M$ into a $0-1$ matrix, which is a matrix with suitable row- and column sums, therefore it is the adjacency matrix of a realization $K$ of the degree sequence $\vec\bd$.
\end{proof}

\section{Erd\H{o}s-R\'enyi random graphs}\label{sec:ER}
\noindent The following statement is a straightforward, easy consequence of Theorem~\ref{th:main}.
\begin{corollary}
If $G$ is a bipartite Erd\H{o}s-R\'enyi random graph on vertex classes of size $n$ and $m$, with edge probability $p(n,m)$, such that
\[ 3\cdot \sqrt{\frac{\log m+\frac12\log 2}{n}}\le p(n,m) \le 1-3\cdot\sqrt{\frac{\log n+\frac12\log 2}{m}}, \]
(or the similar inequalities hold where the roles of $n$ and $m$ are swapped), then the swap Markov chain is rapidly mixing on the bipartite degree sequence of $G$ with probability at least $1-\frac{1}{n}-\frac{1}{m}$.
\end{corollary}
\begin{proof}
  Let $p=p(n,m)$, $\varepsilon_c=\frac13pn$ and $\varepsilon_d=\frac13(1-p)m$. Also, let $c_1=pn-\varepsilon_c$, $c_2=pn+\varepsilon_c$, $d_1=pm-\varepsilon_d$, $d_2=pm+\varepsilon_d$. Equation~\ref{eq:kisebb} holds, we only need to check that
  \begin{align*}
	  4\varepsilon_c\varepsilon_d&\le \left(pn-\varepsilon_c\right)\cdot \left(m-pm-\varepsilon_d\right).
  \end{align*}
  Moreover, by Hoeffding's inequality,
  \begin{align*}
    &\Pr(\text{Equation~\ref{eq:kicsi} does not hold})\le \\
    &\le \Pr\left(\exists v\in V\ |d(v)-pn|>\varepsilon_c\right)+\Pr\left(\exists u\in U\ |d(u)-pm|>\varepsilon_d\right)\le \\
    &\le m\cdot 2e^{-2p^2n/9}+n\cdot 2e^{-2{(1-p)}^2 m/9}\le \frac{1}{m}+\frac{1}{n},
  \end{align*}
  which proves the statement.
\end{proof}
For completeness sake, we also state the respective theorem for directed random graphs.
\begin{corollary}
	If $\vec D$ is a directed Erd\H{o}s-R\'enyi random graph on $n$ vertices with out-edge probability $p(n)$, such that
	\[ 3\cdot \sqrt{\frac{\log n+\frac12\log 2}{n}}+\frac{2}{\sqrt{n}}\le p(n) \le 1-3\cdot\sqrt{\frac{\log n+\frac12\log 2}{n}}-\frac{2}{\sqrt{n}}, \]
	then the swap Markov chain is rapidly mixing on the directed degree sequence of $\vec D$ with probability at least $1-\frac{2}{n}$.
\end{corollary}

\section{Comparing the results on directed degree sequences} \label{sec:compare}

From the provided analysis it is clear that results of Greenhill and Sfragara and our results on directed degree sequences are not completely comparable. However, we can give some regions where Greenhill-Sfragara's and where our results apply. It is easy to show that the GS condition is never satisfied when the average degree $\overline{d} > \frac{n}{16}$ therefore there are regions where the GS result is not applicable while the new results there are. For example, if all degrees are between $\frac{n}{3}+1$ and $\frac{2n}{3}-1$, then  $\overline{d} > \frac{n}{16}$. Still, we can prove rapid mixing of the swap Markov chain for this case. Another consequence is that our results cover all regular degree sequences, but the GS model is not applicable when the average degree is $ > n/16.$

On the other hand, if the degrees are evenly distributed between $1$ and $\frac{n}{32}$, then the GS condition is satisfied. This is a degree sequence for which our theorem cannot be applied.

Generally speaking the Greenhill - Sfragara results are better applicable for degree sequences developed under some scale-free random dynamics (with $\gamma > 2.5$), our new results are better fitted to degree sequences developed under the Erd\H{o}s -- R\'enyi model.

%\section*{\refname}
\bibliographystyle{plain}

\end{document}